\documentclass[a4paper,11pt]{article}
\usepackage{amssymb}

\usepackage[cmex10]{amsmath}
\usepackage{amsfonts}
\usepackage{amsthm}

\usepackage{url}

\usepackage{fixltx2e}

\usepackage{lmodern}
\usepackage[T1]{fontenc}
\usepackage{textcomp}

\newcommand{\eat}[1]{}
\theoremstyle{definition} \newtheorem{theorem}{Theorem}[section]
\theoremstyle{definition} \newtheorem{definition}[theorem]{Definition}
\theoremstyle{definition} \newtheorem{lemma}[theorem]{Lemma}
\theoremstyle{definition} \newtheorem{proposition}[theorem]{Proposition}
\theoremstyle{definition} \newtheorem{corollary}[theorem]{Corollary}
\theoremstyle{definition} 
\theoremstyle{definition} \newtheorem{example}[theorem]{Example}
\theoremstyle{definition} 
\theoremstyle{definition} \newtheorem{remark}[theorem]{Remark}
\theoremstyle{definition} 

\title{(m,n)-Semihyperrings and an Algebra of Fuzzy (m,n)-Semihyperrings}			% used by \maketitle
\author{Syed Eqbal Alam  \footnote{\tt syed.eqbal@ieee.org}, Sultan Aljahdali \footnote{\tt aljahdali@tu.edu.sa} , Nisar Hundewale \footnote{\tt nisar@computer.org}  
    \\
    College of Computers and Information Technology
    \\
Taif University - Taif \\ 
    Kingdom of Saudi Arabia \\
}
\date{}

\begin{document}
\maketitle	

\begin{abstract}
We propose a new class of algebraic structure named as \emph{$(m,n)$-semihyperring} which is a generalization of usual \emph{semihyperring}. We define the basic properties of $(m,n)$-semihyperring like identity elements, weak distributive $(m,n)$-semihyperring, zero sum free, additively idempotent, hyperideals, homomorphism, inclusion homomorphism, congruence relation, quotient $(m,n)$-semihyperring etc. We propose some lemmas and theorems on homomorphism, congruence relation, quotient $(m,n)$-semihyperring, etc and prove these theorems. We further extend it to introduce the relationship between fuzzy sets and $(m,n)$-semihyperrings and propose fuzzy hyperideals and homomorphism theorems on fuzzy $(m,n)$-semihyperrings and the relationship between fuzzy $(m,n)$-semihyperrings and the usual $(m,n)$-semihyperrings.    
\end{abstract}


\section{Introduction}	 	% automatic title!
A semihyperring is essentially a semiring in which addition is a hyperoperation~\cite{davvaz2009}. Semihyperring is in active research for a long time. Vougiouklis~\cite{vougiouklis1990} generalize the concept of hyperring $(\mathcal{R},\oplus, \odot)$ by dropping the reproduction axiom where $\oplus$ and $\odot$ are associative hyper operations and $\odot$ distributes over $\oplus$ and named it as semihyperring.
Chaopraknoi, Hobuntud and Pianskool ~\cite{chaopraknoi2008} studied semihyperring with zero.  
 Davvaz and Poursalavati~\cite{bdavvaz1999} introduced the matrix representation of polygroups over hyperring and also over semihyperring. Semihyperring and its ideals are studied by Ameri and Hedayati ~\cite{ameriHey2007}.
 
 Zadeh~\cite{lzadeh1965} introduced the notion of a fuzzy set that is used to formulate some of the basic concepts of algebra. It is extended to fuzzy hyperstructures, now a days fuzzy hyperstructure is a fascinating research area. Davvaz introduced the notion of fuzzy subhypergroups in~\cite{bdavvaz1999}, Ameri and Nozari~\cite{ameri2011} introduced fuzzy regular relations and fuzzy strongly regular relations of fuzzy hyperalgebras and also established a connection between fuzzy hyperalgebras and algebras. Fuzzy subhypergroup is also studied by Cristea~\cite{irinaC2009}. Fuzzy hyperideals of semihyperrings are studied by ~\cite{davvaz2009,davvazdudek2009,ameriH2010}.
 
 The generalization of Krasner hyperring is introduced by Mirvakili and Davvaz ~\cite{mirvakili2010} that is named as Krasner $(m,n)$ hyperring. In ~\cite{davvaz2010} Davvaz studied the fuzzy hyperideals of the Krasner  $(m,n)$-hyperring. Generalization of hyperstructures are also studied by ~\cite{davvazVou2006, davvaz2009, bdavvazpcor2009, fotea2009}. 
 
In this paper, we introduce the notion of the generalization of usual semihyperring and called it as $(m,n)$-semihyperring and set fourth some of its properties, we also introduce fuzzy $(m,n)$-semihyperring and its basic properties and the relation between fuzzy $(m,n)$-semihyperring and its associated $(m,n)$-semihyperring. 

The paper is arranged in the following fashion:

Section ~\ref{prelim} describes the notations used and the general conventions followed. Section ~\ref{def_semihyp} deals with the definitions of $(m,n)$-semihyperring, weak distributive $(m,n)$-semihyperring, hyperadditive and multiplicative identity elements, zero, zero sum free, additively idempotent and some examples of $(m,n)$-semihyperrings. 

Section~\ref{prop_semhyp} describes the properties of $(m,n)$-semihyperring. This section deals with the definitions of hyperideals, homomorphism, congruence relation, quotient of $(m,n)$-semihyperring and also the theorems based on these definitions.   

Section \ref{fuuzy_sem} deals with the fuzzy $(m,n)$-semihyperrings, fuzzy hyperideals and homomorphism theorems on $(m,n)$-semihyperrings and fuzzy $(m,n)$-semihyperrings.  

%Section \ref{concl} states the conclusion of the paper.

\section{Preliminaries} \label{prelim}

Let \(\mathcal{H}\) be a non-empty set and \(\mathcal{P^{*}}(\mathcal{H})\) be the set of all non-empty subsets of \(\mathcal{H}\). A hyperoperation on \(\mathcal{H}\) is  a map \(\sigma\) : \(\mathcal{H} \times \mathcal{H}\)$\rightarrow$ \(\mathcal{P^{*}}(\mathcal{H})\) and the couple $(\mathcal{H} ,\sigma)$ is called a \emph{hypergroupoid}. If $A$ and $B$ are non-empty subsets of \(\mathcal{H}\), then we denote 
$A \sigma B$ = ${\displaystyle \bigcup_{ a \in A,b \in B} a \sigma b}$,   $x \sigma A$= $\{x\} \sigma A$ and $A \sigma x$= $A \sigma \{x\}$.  

Let $\mathcal{H}$ be a non-empty set, $\mathcal{P^{*}}$ be the set of all non-empty subsets of $\mathcal{H}$ and a mapping $f:\mathcal{H}^m \rightarrow \mathcal{P^{*}}(\mathcal{H})$ is called an \emph{$m$-ary hyperoperation} and $m$ is called the \emph{arity of hyperoperation}~\cite{davvazVou2006}.

A hypergroupoid $(\mathcal{H} ,\sigma)$ is called a \emph{semihypergroup} if for all $x,y,z \in \mathcal{H}$ we have $(x \sigma y) \sigma z= x \sigma (y \sigma z)$ which means that
\[{\displaystyle \bigcup_{ u \in x \sigma y} u \sigma z} = {\displaystyle \bigcup_{ v \in y \sigma z} x \sigma v}.\] 

Let $f$ be an $m$-ary hyperoperation on \(\mathcal{H}\) and $A_1,A_2,\ldots A_m$ subsets of \(\mathcal{H}\). We define 
\[f(A_1,A_2,\ldots, A_m) ={\displaystyle \bigcup_{x_i \in A_i} f(x_1,x_2,\ldots, x_m)}\]
for all $1 \leq i \leq m$.

\begin{definition} 
$(\mathcal{H},\oplus,\otimes)$ is a semihyperring  which satisfies the following axioms:

\begin{itemize}

 \item[(i)]  $(\mathcal{H},\oplus)$ is a semihypergroup
 \item[(ii)] $(\mathcal{H},\otimes)$ is a semigroup and 
 \item[(iii)]$\otimes$ distributes over $\oplus$  

$x \otimes (y\oplus z) =(x \otimes y) \oplus (x \otimes z)$ and 

$(y \oplus z ) \otimes x =(y \otimes x) \oplus (z \otimes x)$
for all $x,y,z \in \mathcal{H}$~\cite{chaopraknoi2008}.

\end{itemize}
\end{definition}

\begin{example} 
Let $(\mathcal{H},+,\times)$ be a semiring, we define
\begin{itemize}
 \item[(i)] $x \oplus y = <x , y>$
\item[(ii)] $x \otimes y = x \times y$

Then $(\mathcal{H},\oplus,\otimes)$ is a semihyperring.
\end{itemize}
\end{example}

An element $0$ of a semihyperring $(\mathcal{H},\oplus,\otimes)$ is called a \emph{zero} of $(\mathcal{H},\oplus,\otimes)$ if $x \oplus 0$ =  $0 \oplus x$ = $\lbrace x \rbrace$ and  $x \otimes 0 = 0 \otimes x = 0$~\cite{chaopraknoi2008}.

The set of integers is denoted by $\mathbb{Z}$, with $\mathbb{Z}_+$
and $\mathbb{Z}_-$ denoting the sets of positive integers and negative
integers respectively. Elements of the set $\mathcal{H}$ are denoted by $x_i,y_i$ where $i \in \mathbb{Z}_+$.

We use following general convention as followed
by~\cite{eqbal2010,dudekV1996,dudek2001,davvazdudek2009}:

The sequence $x_{i}$, $x_{i+1}$, \ldots, $x_{m}$ is denoted by
$x_i^{m}$. 

The following term:
\begin{equation} \label{defeqnfm}
f(x_{1}, \ldots ,x_{i},y_{i+1}, \ldots ,y_{j},z_{j+1} , \ldots , z_{m}) 
\end{equation}
is represented as: 
\begin{equation} \label{defeqnoff}
f(x_1^{i},y_{i+1}^{j},z_{j+1}^{m})
\end{equation}

In the case when $y_{i+1} = \ldots = y_{j} = y$, then \eqref{defeqnoff}
is expressed as:
\[f(x_{1}^i , \stackrel{(j-i)}{y} , z_{j+1}^m)\] 

\begin{definition} 

A non-empty set \(\mathcal{H}\) with an $m$-ary hyperoperation $f:\mathcal{H}^m \rightarrow \mathcal{P^{*}}(\mathcal{H})$ is called an \emph{$m$-ary hypergroupoid} and is denoted as \((\mathcal{H},f)\). An $m$-ary hypergroupoid \((\mathcal{H},f)\) is called an \emph{$m$-ary semihypergroup} if and only if the following associative axiom holds :
\[f(x_{1}^{i},f(x_i^{m+i-1}),x_{m+i}^{2m-1}) = f(x_{1}^{j},f(x_j^{m+j-1}),x_{m+j}^{2m-1})\]
for all $i,j\in \{1,2,\ldots,m\}$ and $x_1,x_2,\ldots,x_{2m-1} \in \mathcal{H}$~\cite{davvazVou2006}.

\end{definition}

\begin{definition} 

Element  $e$ is called \emph{identity element} of hypergroup \((\mathcal{H},f)\) if 

$x \in f(\underbrace{e,\ldots,e}_{i-1},x,\underbrace{e,\ldots,e}_{n-i})$ 

for all $x \in \mathcal{H}$ and $1 \leq i \leq n$~\cite{davvazVou2006}.

\end{definition}

\begin{definition} \label{n_ary_groupoid}

 A non empty set $\mathcal{H}$ with an $n$-ary operation $g$ is called an \emph{$n$-ary groupoid} and is denoted by ($\mathcal{H}$,$g$)~\cite{dudek2001}.

\end{definition}

\begin{definition} \label{def_semigroup}

An $n$-ary groupoid $(\mathcal{H},g)$ is called an \emph{$n$-ary semigroup} if $g$ is associative, i.e.,
 \[g(x_{1}^{i},g(x_i^{n+i-1}),x_{n+i}^{2n-1}) = g(x_{1}^{j},g(x_j^{n+j-1}),x_{n+j}^{2n-1})\]
for all $i,j\in \{1,2,\ldots,n\}$ and $x_1,x_2,\ldots,x_{2n-1} \in \mathcal{H}$ ~\cite{dudek2001}.

\end{definition}

\section{Definitions and Examples of $(m,n)$-Semihyperring} \label{def_semihyp}

\begin{definition} \label{def_sem}

$(\mathcal{H},f,g)$ is an $(m,n)$-semihyperring  which satisfies the following axioms:

\begin{itemize}
 \item[(i)]  $(\mathcal{H},f)$ is a $m$-ary semihypergroup
 \item[(ii)] $(\mathcal{H},g)$ is an $n$-ary semigroup and 
 \item[(iii)] $g$ is distributive over $f$ i.e.,

$g(x_1^{i-1}, f(a_{1}^{m}), x_{i+1}^{n})$= $f(g(x_1^{i-1}, a_{1},x_{i+1}^{n}) , \ldots , g(x_1^{i-1}, a_{m}, x_{i+1}^{n})).$
\end{itemize}

\end{definition}

\begin{remark} \label{rem_sem}
An $(m,n)$-semihyperring is called \emph{weak distributive} if it satisfies Definition \ref{def_sem} (i), (ii) and the following: 
 
$g(x_1^{i-1}, f(a_{1}^{m}), x_{i+1}^{n})$ $\subseteq$ $f(g(x_1^{i-1}, a_{1},x_{i+1}^{n}) , \ldots , g(x_1^{i-1}, a_{m}, x_{i+1}^{n})).$

\end{remark}
Remark \ref{rem_sem} is generalization of ~\cite{ameri2005}.

\eat{
\begin{definition} \label{commutative_sem}
An $(m,n)$-semihyperring $(\mathcal{H},f,g)$ is called \emph{commutative} if 

$g(x_1 , x_2 , \ldots , x_n)$ = $g(x_{\eta(1)} , x_{\eta(2)} ,
  \ldots , x_{\eta(n)})$

  for every permutation $\eta$ of $\{1,2, \ldots, n\}$ ~\cite{timm1967, eqbal2010}, $\forall$ $x_1$,$ x_2$,$ \ldots, x_n$ $\in
  \mathcal{H}$.

\end{definition}
}

\begin{example} \label{sem_exp1}

Let $\mathbb{Z}$ be the set of all integers. Let the binary hyperoperation $\oplus$ and an $n$-ary operation $g$ on $\mathbb{Z}$ which are defined as  follows:

$x_1 \oplus x_2$ = $\lbrace x_1, x_2 \rbrace$ and

$g(x_1,x_2,\ldots, x_n) = \prod_{i=1}^n x_i$.

Then $(\mathbb{Z},\oplus,g)$ is called a \emph{$(2,n)$-semihyperring}.

\end{example}

Example ~\ref{sem_exp1} is generalization of Example 1 of ~\cite{davvaz2009}.  

\begin{definition} \label{def_identityelem}

Let $e$ be the \textit{hyper additive identity element} of hyperoperation $f$ and $e'$ be \textit{multiplicative identity element} of operation $g$ then
          \[x \in f(\underbrace{e , \ldots ,e}_{i-1} , x , \underbrace{e , \ldots , e}_{m-i})\] 
for all $x \in \mathcal{H}$ and $1 \leq i \leq m$ and 
\[y=g(\underbrace{e', \ldots , e'}_{j-1},y,\underbrace{e' , \ldots , e'}_{n-j})\] 
for all $y \in \mathcal{H}$ and $1 \leq j \leq n$.

\end{definition}

\begin{definition} 

An element $\mathbf{0}$ of an $(m,n)$-semihyperring $(\mathcal{H},f,g)$ is called a \emph{zero} of $(\mathcal{H},f,g)$ if 
\[f(\underbrace{\mathbf{0},\ldots,\mathbf{0}}_{m-1},x) = f(x,\underbrace{\mathbf{0},\ldots,\mathbf{0}}_{m-1})=x\] for all $x \in \mathcal{H}$. 
\[g(\underbrace{\mathbf{0},\ldots,\mathbf{0}}_{n-1},y) = g(y,\underbrace{\mathbf{0},\ldots,\mathbf{0}}_{n-1}) = \mathbf{0}\] for all $y \in \mathcal{H}$. 

\end{definition}

%%%%%%%%%%%%%%%%%%%%%below New 27Jan%%%%%%%%%%%%%%%%%%
\begin{remark}
Let $(\mathcal{H},f,g)$ be an $(m,n)$-semihyperring and $e$ and $e'$ be hyper additive identity and multiplicative identity elements respectively, then we can obtain the additive hyper operation and multiplication as follows:
\[ \langle x,y \rangle  = f(x,\underbrace{e,\ldots, e}_{m-2}, y)\]
and $x \times y = g(x,\underbrace{e',\ldots,e'}_{n-2},y)$
for all $x,y \in \mathcal{H}$. 
\end{remark}

%%%%%%%%%%%new%
\begin{definition}  

Let $(\mathcal{H},f,g)$ be an $(m,n)$-semihyperring.
\begin{itemize} 
\item[(i)]$(m,n)$-semihyperring $(\mathcal{H},f,g)$ is called \emph{zero sum free} if and only if $\mathbf{0} \in f(x_1,x_2,\ldots,x_m)$ implies $x_1=x_2=,\ldots,=x_m=\mathbf{0}$.
\item[(ii)] $(m,n)$-semihyperring $(\mathcal{H},f,g)$ is called \emph{additively idempotent} if $(\mathcal{H},f)$ be a $m$-ary semihypergroup, i.e. if $f(x,x,\ldots,x) \in x$.
\end{itemize}

\end{definition}

\section{Properties of $(m,n)$-Semihyperring} \label{prop_semhyp}

\begin{definition} \label{def_ideal}
Let $(\mathcal{H},f,g)$ be an $(m,n)$-semihyperring. 
\begin{itemize}
\item[(i)]
An $m$-ary sub-semihypergroup $\mathcal{R}$ of $\mathcal{H}$ is called an \emph{$(m,n)$-sub-semihyperring} of $\mathcal{H}$ if $g(a_1^n) \in \mathcal{R}$, for all $a_1,a_2,\ldots,a_n \in \mathcal{R}$.
\item[(ii)]
An $m$-ary sub-semihypergroup $\mathcal{I}$ of $\mathcal{H}$ is called 
\begin{itemize}
\item[(a)]
a \emph{left hyperideal} of $\mathcal{H}$ if $g(a_1^{n-1},i) \in \mathcal{I}, \forall a_1,a_2,\ldots,a_{n-1} \in \mathcal{H}$ and $i \in \mathcal{I}$.
\item[(b)]
a \emph{right hyperideal} of $\mathcal{H}$ if $g(i,a_1^{n-1}) \in \mathcal{I}, \forall a_1,a_2,\ldots,a_{n-1} \in \mathcal{H}$ and $i \in \mathcal{I}$.

If $\mathcal{I}$ is both left and right hyperideal then it is called as an \emph{hyperideal} of $\mathcal{H}$.
\item[(c)] 
a left hyperideal $\mathcal{I}$ of an $(m,n)$-semihyperring of $\mathcal{H}$ is called \emph{weak left hyperideal} of $\mathcal{H}$ if for $i \in \mathcal{I}$ and $x_1,x_2,\ldots,x_{m-1}$ $\in \mathcal{H}$ then

$f(i,x_1^{m-1}) \subseteq \mathcal{I}$ or $f(x_1^{m-1},i) \subseteq \mathcal{I}$ implies $x_1,x_2,\ldots,x_{m-1} \in \mathcal{I}$.
 
\end{itemize}

\end{itemize}

\end{definition} 
Definition~\ref{def_ideal} is generalization of ~\cite{ameri1999}.

\begin{proposition} \label{prop_hyp}

A left hyperideal of an $(m,n)$-semihyperring is an $(m,n)$-sub-semihyperring.

\end{proposition}
 
\begin{definition} 

Let $(\mathcal{H},f,g)$ and $(\mathcal{S},f',g')$ be two $(m,n)$-semihyperrings. The mapping $\sigma: \mathcal{H} \rightarrow \mathcal{S}$ is called a \emph{homomorphism} if following condition is satisfied for all $x_1,x_2,\ldots,x_m$, $y_1,y_2,\ldots,y_n \in \mathcal{H}$.

\[\sigma(f(x_1,x_2,\ldots,x_m))=f'(\sigma(x_1),\sigma(x_2),\ldots,\sigma(x_m))\] and 

\[\sigma(g(y_1,y_2,\ldots,y_n))=g'(\sigma(y_1),\sigma(y_2),\ldots,\sigma(y_n)).\]

\end{definition}

\begin{remark} \label{rem_inc_homom}

Let $(\mathcal{H},f,g)$ and $(\mathcal{S},f',g')$ be two $(m,n)$-semihyperrings. The mapping $\sigma: \mathcal{H} \rightarrow \mathcal{S}$ for all $x_1,x_2,\ldots,x_m$, $y_1,y_2,\ldots,y_n \in \mathcal{H}$ is called an \emph{inclusion homomorphism} if following relations hold: 

\[\sigma(f(x_1,x_2,\ldots,x_m)) \subseteq f'(\sigma(x_1),\sigma(x_2),\ldots,\sigma(x_m))\] and 

\[\sigma(g(y_1,y_2,\ldots,y_n)) \subseteq g'(\sigma(y_1),\sigma(y_2),\ldots,\sigma(y_n))\]

\end{remark}

Remark \ref{rem_inc_homom} is generalization of ~\cite{bdavvaz1999}.

\begin{theorem} \label{homomorphism} 
Let $(\mathcal{R}, f, g)$, $(\mathcal{S}, f', g')$ and $(\mathcal{T} , f'', g'')$ be $(m, n)$-semihyperrings.
If mappings $ \sigma : (\mathcal{R}, f, g) \longrightarrow (\mathcal{S}, f', g')$ and
$\delta : (S, f', g') \longrightarrow (T , f'', g'')$ are homomorphisms, then  $\sigma \circ \delta : (R, f, g) \longrightarrow (T , f'', g'')$ is also a homomorphism.
\end{theorem}

\begin{proof}
Omitted as obvious.
\end{proof}

\begin{definition} \label{def_congr}

Let $\cong$ be an equivalence relation on the $(m,n)$-semihyperring $(\mathcal{H}, f ,g)$ and $A_i$ and $B_i$ be the subsets of $\mathcal{H}$ for all $1 \leq i \leq m$. We define $A_i \cong B_i$ for all $a_i \in A_i$ there exists $b'_i \in B_i$ such that $a_i \cong b'_i$ holds true and for all $b_i \in B_i$ there exists $a'_i \in A_i$ such that $a'_i \cong b_i$ holds true~\cite{mksen2010}.

  An equivalence relation $\cong$ is called a \emph{congruence relation} on $\mathcal{H}$ if following hold: 

\begin{itemize}

\item[(a)] for all $a_1,a_2,\ldots,a_m$, $b_1,b_2,\ldots,b_m$ $\in \mathcal{H}$; if $\lbrace a_i \rbrace \cong \lbrace b_i \rbrace$ then $\lbrace f(a_1^m) \rbrace \cong \lbrace f(b_1^m) \rbrace$, where $1 \leq i \leq m$ and,

\item[(b)] for all $x_1,x_2,\ldots,x_n$, $y_1,y_2,\ldots,y_n \in \mathcal{H}$; if $x_j \cong y_j$ then $g(x_1^n) \cong g(y_1^n)$, where  $1 \leq j \leq n$ ~\cite{Burris1981}.
\end{itemize}

\end{definition}

\begin{lemma} \label{lemma_cong}

  Let ($\mathcal{H}, f, g)$ be an $(m,n)$-semihyperring and $\cong$ be the congruence relation on $\mathcal{H}$ then    

\begin{itemize}

\item[(i)]  
if $\lbrace x\rbrace \cong \lbrace y\rbrace$ then 
  \[ \lbrace f(x, a_{1}^{m-1}) \rbrace \cong \lbrace f(y, a_{1}^{m-1}) \rbrace \]   
  for all $x,y,a_1,a_2,\ldots,a_{m} \in \mathcal{H}$

\item[(ii)]
if $x \cong y$ then following holds:
 \[ g(a_{1}^{i-1},x,a_{i+1}^{n}) \cong g(a_{1}^{i-1},y,a_{i+1}^{n})\]   
  for all $x,y,a_1,a_2,\ldots,a_{n} \in \mathcal{H}$

\end{itemize}

\end{lemma}

\begin{theorem}
 Let ($\mathcal{H}, f, g)$ be an $(m,n)$-semihyperring and $\cong$ be the congruence relation on $\mathcal{H}$. Then if $\lbrace a_i \rbrace \cong \lbrace b_i \rbrace$ and $\lbrace x_j \rbrace \cong \lbrace y_j \rbrace$ for all $a_i,b_i,x_j,y_j \in \mathcal{H}$ and $i,j \in \lbrace 1,m\rbrace$
then the following is obtained:

 for all $1 \leq k \leq m$ 
\[\lbrace f(a_1^k,x_{k+1}^m)\rbrace \cong \lbrace f(b_1^k,y_{k+1}^m) \rbrace\]

\end{theorem}

\begin{proof}
Can be proved similar to Lemma \ref{lemma_cong}.  
\end{proof}

\begin{definition} 

Let $\cong$ be a congruence on $\mathcal{H}$. Then the \emph{quotient} of $\mathcal{H}$ by $\cong$, written as $\mathcal{H}/ \cong$, is the algebra whose universe is $\mathcal{H}/ \cong$ and whose fundamental operation satisfy 

$f^{\mathcal{H}/\cong}(x_1,x_2,\ldots,x_m)$ = $ f^{\mathcal{H}}(x_1,x_2,\ldots,x_m)/ \cong $

where $x_1,x_2,\ldots,x_m \in \mathcal{H}$~\cite{Burris1981}.

\end{definition}

\begin{theorem} 

Let $(\mathcal{H},f,g)$ be an $(m,n)$-semihyperring  and $\cong$ be the equivalence relation and strongly regular on $\mathcal{H}$ then $(\mathcal{H}/ \cong, f, g)$ is also an $(m,n)$-semihyperring.

\end{theorem} 

\begin{definition} \label{def_natural_map}

Let $(\mathcal{H},f,g)$ be an $(m,n)$-semihyperring and $\cong$ be the congruence relation. The natural map $v_{\cong}:\mathcal{H} \longrightarrow \mathcal{H} / \cong$ is defined by $v_{\cong} (a_i) = a_i/\cong$ and  $v_{\cong} (b_j) = b_j/\cong$ where $a_i, b_j \in \mathcal{H}$ for all $1 \leq i \leq m$, $1 \leq j \leq n$.

\end{definition}

\begin{theorem} \label{theorem_quot}

Let $\rho$ and $\sigma$ be two congruence relations on
 $(m,n)$-semihyperring $(\mathcal{H},f,g)$ such that $\rho \subseteq \sigma$. Then $\sigma / \rho = \lbrace (\rho(x), \rho(y)) \in \mathcal{H}/\rho \times \mathcal{H}/\rho : (x,y) \in \sigma \rbrace$ is a congruence on $\mathcal{H}/ \rho$ and $(\mathcal{H}/\rho) \diagup (\sigma/\rho) \cong \mathcal{H}/\sigma.$  

\end{theorem}

\begin{proof}

Similar to \cite{skar2007}, we can deduce that $\sigma/\rho$ is an equivalence relation on $\mathcal{H}/\rho$. 
Suppose $(a_i\rho)(\sigma/\rho)(b_i\rho)$ for all $1 \leq i \leq m$ and $(c_j\rho)(\sigma/\rho)(d_j\rho)$ for all $1 \leq j \leq n$. Since $\sigma$ is congruence on $\mathcal{H}$ therefore $f(a_1^m) \sigma f(b_1^m)$ and $g(c_1^n) \sigma g(d_1^n)$ which implies $f(a_1^m)\rho (\sigma/\rho) f(b_1^m)\rho$ and $g(c_1^n)\rho (\sigma/\rho) g(d_1^n)\rho$ respectively, therefore $\sigma/\rho$ is a congruence on $\mathcal{H}/\rho$.
\qedhere
\end{proof} 

\begin{theorem} \label{theorem_natural_map}

The natural map from an $(m,n)$-semihyperring $(\mathcal{H},f,g)$ to the quotient $(\mathcal{H} / \cong,f,g)$ of the $(m,n)$-semihyperring is an onto homomorphism. 

\end{theorem}

Definition \ref{def_natural_map} and Theorem \ref{theorem_natural_map} is generalization of \cite{Burris1981}.

\begin{proof}

let $\cong$ be the congruence relation on $(m,n)$-semihyperring $(\mathcal{H},f,g)$ and the natural map be 
$v_{\cong}:\mathcal{H} \longrightarrow \mathcal{H} / \cong$.
For all $a_i \in \mathcal{H}$, where $1 \leq i \leq m$ following holds true:

 \[ v_{\cong} f^{\mathcal{H}}(a_1,a_2,\ldots,a_m) = f^{\mathcal{H}}(a_1,a_2, \dots, a_m)/ {\cong} \]

  \[= f^{\mathcal{H}/ \cong}(a_1/ \cong, a_2/ \cong, \dots, a_m/ \cong)\]

 \[= f^{\mathcal{H}/ {\cong}}(v_{\cong} a_1, v_{\cong} a_2, \ldots, v_{\cong} a_m)\]

In a similar fashion we can deduce for $g$, for all $b_j \in \mathcal{H}$, where $1 \leq j \leq n$: 

 \[ v_{\cong} g^{\mathcal{H}}(b_1,b_2,\ldots,b_n) = g^{\mathcal{H}}(b_1,b_2, \ldots, b_n)/ {\cong} \]

  \[= g^{\mathcal{H}/ {\cong}}(b_1/ {\cong}, b_2/ {\cong}, \ldots, b_n/ {\cong})\]

 \[ = g^{\mathcal{H}/ {\cong}}(v_{\cong} b_1, v_{\cong} b_2, \ldots, v_{\cong} b_n) \]

 So $v_{\cong}$ is onto homomorphism.

Proof is similar to \cite{Burris1981}. \qedhere

\end{proof}

\section{Fuzzy $(m,n)$-Semihyperring} \label{fuuzy_sem}

Let $\mathcal{R}$ be a non-empty set. Then

\begin{itemize}

\item[(i)] A fuzzy subset of $\mathcal{R}$ is a function $\mu : \mathcal{R} \longrightarrow [0, 1]$;
\item[(ii)] For a fuzzy subset $\mu$ of $\mathcal{R}$ and $t \in [0, 1]$, the set $\mu_{t} =\{ x \in \mathcal{R}  \mid \mu(x) \geq t\}$ is called the \emph{level subset} of $\mu$~\cite{davvaz2010, lzadeh1965, xueling2008, davvaz2009}.

\end{itemize}

\begin{definition}
 A fuzzy subset $\mu$ of an $(m,n)$-semihyperring $(\mathcal{H},f,g)$ is called a \emph{fuzzy $(m,n)$-sub-semihyperring} of $\mathcal{H}$ if following hold true:
 
 \begin{itemize}

 \item[(i)]  min$\{\mu(x_1), \mu(x_2), \ldots, \mu(x_m)\} \leq {\displaystyle {inf}_{z \in f(x_1,x_2,\ldots, x_m)} \mu(z)}$, 

for all $x_1,x_2, \ldots,x_m \in \mathcal{H}$
 \item[(ii)]  min$\{\mu(x_1), \mu(x_2), \ldots, \mu(x_n)\} \leq$ $ \mu(g(x_1,x_2,\ldots,x_n))$,
 
 for all $x_1,x_2, \ldots,x_n \in \mathcal{H}$.
 
\end{itemize}
\end{definition}

\begin{definition}
A fuzzy subset $\mu$ of an $(m,n)$-semihyperring $(\mathcal{H},f,g)$ is called a \emph{fuzzy hyperideal} of 
$\mathcal{H}$ if the following hold true:

\begin{itemize}
 \item[(i)] min$\{\mu(x_1), \mu(x_2), \ldots, \mu(x_m)\} \leq$ ${\displaystyle {inf}_{z \in f(x_1,x_2,\ldots, x_m)} \mu(z)}$,
 
  for all $x_1,x_2, \ldots,x_m \in \mathcal{H},$

\item[(ii)] $\mu(x_1) \leq \mu(g(x_1,x_2,\ldots,x_n))$, for all $x_1,x_2, \ldots,x_n \in \mathcal{H}$,
\item[(iii)] $\mu(x_2) \leq \mu(g(x_1,x_2,\ldots,x_n))$, for all $x_1,x_2, \ldots,x_n \in \mathcal{H}$,

 $\hspace{1.6in} \vdots$

\item[(iv)] $\mu(x_n) \leq \mu(g(x_1,x_2,\ldots,x_n))$, for all $x_1,x_2, \ldots,x_n \in \mathcal{H}$.
\end{itemize}
\end{definition}

\begin{theorem} \label{fuzzy_sub_theo}
A fuzzy subset $\mu$ of an $(m,n)$-semihyperring $(\mathcal{H},f,g )$ is a fuzzy hyperideal if and only if every non-empty level subset is a hyperideal of $\mathcal{H}$. 
\end{theorem}

\begin{proof}
Suppose subset $\mu$ is a fuzzy hyperideal of $(m,n)$-semihyperring $(\mathcal{H},f,g)$ and $\mu_{t}$ is a level subset of $\mu$.
 
If $x_1,x_2,\ldots,x_m \in \mu_{t}$ for some $t \in [0,1]$ then from the definition of level set, we can deduce the following:

 \[\mu(x_1)\geq t, \mu(x_2)\geq t, \ldots, \mu(x_m)\geq t.\] 

Thus, we say that:

\[min\{\mu(x_1), \mu(x_2), \ldots, \mu(x_m)\} \geq t\] 

Thus:

\[{\displaystyle {inf}_{z \in f(x_1,x_2,\ldots, x_m)} \mu(z)} \geq min\{\mu(x_1), \mu(x_2), \ldots, \mu(x_m)\}\geq t.\]
 So, we get the following: 

$\mu(z) \geq t$, for all $z \in f(x_1,x_2, \ldots, x_m)$.

Therefore, $f(x_1,x_2, \ldots, x_m) \subseteq \mu_{t}$. 

Again, suppose that $x_1,x_2, \ldots, x_n \in \mathcal{H}$ and $x_i \in \mu_{t}$, where $1 \leq i \leq n$. Then, we find that $\mu(x_i) \geq t$.

So, we obtain the following:
\[t \leq \mu_{x_i} \leq \mu(g(x_1,x_2,\ldots,x_n)) \longrightarrow g(x_{1}^{i-1}, \mu_{t},x_{i+1}^{n}) \subseteq \mu_{t}\]

Thus, we find that $\mu_{t}$ is a hyperideal of $\mathcal{H}$.    

On the other hand, suppose that every non-empty level subset $\mu_{t}$ is a hyperideal of $\mathcal{H}$
is a hyperideal of $\mathcal{H}$.

Let $t_{0}$ = min\{ $\mu_{x_1}, \mu_{x_2}, \ldots, \mu_{x_n}$\}, for all $x_1,x_2, \ldots, x_n \in \mathcal{H}$.

Then, we obtain the following:
\[\mu(x_1) \geq t_{0}, \mu(x_2) \geq t_{0}, \ldots, \mu(x_n) \geq t_{0}\]
Thus, 
\[ x_1, x_2, \ldots, x_n \in \mu_{t_0}\]
We can also obtain that:
\[ f(x_1, x_2, \ldots, x_m) \subseteq \mu_{t_{0}}.\]
Thus,
\[ min\{ \mu(x_1),\mu(x_2), \ldots, \mu(x_m)\}= t_{0} \leq {\displaystyle {inf}_{z \in f(x_1,x_2,\ldots, x_m)} \mu(z)}. \] 

Again, suppose that $\mu(x_1)=t_{1}$. Then $x \in \mu_{t_{1}}$. 

So, we obtain:
\[ g(x_1,x_2, \ldots, x_{n}) \in \mu_{t_{1}} \longrightarrow t_{1} \leq \mu(g(x_1,x_2, \ldots, x_n)) \]
Thus, $\mu(x_1) \leq \mu(g(x_1,x_2, \ldots, x_n))$.

Similarly, we obtain $\mu(x_i) \leq \mu(g(x_1,x_2, \ldots, x_n))$, for all $i \in \{1,n\}$.

Thus, we can check all the conditions of the definition of fuzzy hyperideal.  

This proof is a generalization of \cite{davvaz2009}.\qedhere 
 
\end{proof}

  Theorem~\ref{fuzzy_sub_theo} is a generalization of ~\cite{hedayati2005, ameriH2010, davvaz2009}. 

Jun, Ozturk and Song~\cite{young2004} have proposed a similar theorem on hemiring.

\begin{theorem} \label{fuzzy_hyp_ideal}
Let $\mathcal{I}$ be a non-empty subset of an $(m,n)$-semihyperring $(\mathcal{H},f,g)$. Let $\mu_{I}$ be a fuzzy set defined as follows:

\[
\mu_{I}(x) = \left\{ \begin{array}{rl}
 s &\mbox{ if $x \in \mathcal{I}$}, \\
  t &\mbox{ otherwise},
       \end{array} \right.
\]

where $0 \leq t<s \leq 1$. Then $\mu_{I}$ is a fuzzy left hyper ideal of $\mathcal{H}$ if and only if $\mathcal{I}$ is a left hyper ideal of $\mathcal{H}$.

\end{theorem}

Following Corollary~\ref{coro_davvaz} is generalization of ~\cite{davvaz2009}.

\begin{corollary} \label{coro_davvaz}
Let $\mu$ be a fuzzy set and its upper bound be $t_0$ of an $(m,n)$-semihyperring $(\mathcal{H},f,g)$. Then 
the following are equivalent:

\begin{itemize}
\item[(i)] $\mu$ is a fuzzy hyperideal of $\mathcal{H}$.
\item[(ii)] Every non-empty level subset of $\mu$ is a hyperideal of $\mathcal{H}$.
\item[(iii)] Every level subset $\mu_{t}$ is a hyperideal of $\mathcal{H}$ where $t \in [0,t_{0}]$.
\end{itemize}  

\end{corollary}  
 
 \begin{definition} \label{def_lv}
Let $(\mathcal{R},f',g')$ and $(\mathcal{S},f'',g'')$ be fuzzy $(m,n)$-semihyperrings and $\varphi$ be a map from $\mathcal{R}$ into $\mathcal{S}$. Then $\varphi$ is called \emph{homomorphism} of fuzzy $(m,n)$-semihyperrings if following hold true:

\[\varphi(f'(x_1,x_2,\ldots,x_m)) \leq f''(\varphi(x_1),\varphi(x_2),\ldots,\varphi(x_m))\]
and
\[\varphi(g'(y_1,y_2,\ldots,y_n)) \leq g''(\varphi(y_1),\varphi(y_2),\ldots,\varphi(y_n))\]
   
for all $x_1,x_2,\ldots,x_m,y_1,y_2,\ldots,y_n \in \mathcal{R}.$
 \end{definition}

\begin{theorem} \label{th_lv1}
Let $(\mathcal{R},\mu_{f'},\mu_{g'})$ and $(\mathcal{S},\mu_{f''},\mu_{g''})$ be two fuzzy $(m,n)$-semihyperrings and  $(\mathcal{R},f',g')$ and $(\mathcal{S},f'',g'')$ be associated $(m,n)$-semihyperring. If $\varphi:\mathcal{R}\longrightarrow \mathcal{S}$ is a homomorphism of fuzzy $(m,n)$-semihyperrings, then $\varphi$ is homomorphism of the associated $(m,n)$-semihyperrings also.
\end{theorem}

Definition~\ref{def_lv} and Theorem~\ref{th_lv1} are similar to the one proposed by Leoreanu-Fotea~\cite{fotea2009} on fuzzy $(m,n)$-ary hyperrings and $(m,n)$-ary hyperrings and Ameri and Nozari~\cite{ameri2011} proposed a similar Definition and Theorem on hyperalgebras.    

\eat{
\section{Conclusion} \label{concl}
We proposed a new class of algebraic structure $(m,n)$-semihyperring. 
}

\end{document}